\documentclass[11pt]{amsart}

\usepackage{amsfonts, amstext, amsmath, amsthm, amscd, amssymb} 
\usepackage{mathtools}

\usepackage{float} 
\usepackage{graphics}
\usepackage{caption}
\usepackage{subcaption}
\usepackage{import} 
\usepackage[abs]{overpic}
\usepackage{physics}  
\usepackage{microtype}

\usepackage[hidelinks,pagebackref,pdftex]{hyperref}

\usepackage[dvipsnames]{xcolor}
\definecolor{MyCyan}{HTML}{00F9DE}
\usepackage{enumitem} 
\setlist{nolistsep,leftmargin=*}
\usepackage{transparent}

\usepackage{marginnote}
\long\def\@savemarbox#1#2{\global\setbox#1\vtop{\hsize\marginparwidth 
  \@parboxrestore\tiny\raggedright #2}}
\newcommand{\RR}{\mathbb{R}}  % Reals
\newcommand{\HH}{\mathbb{H}}  % Hyperbolic
\newcommand{\ZZ}{\mathbb{Z}}  % Integers
\newcommand{\QQ}{\mathbb{Q}}  % Rationals

\renewcommand{\SS}{\mathbb{S}}

\renewcommand{\setminus}{{\smallsetminus}}

\newcommand{\from}{\colon\thinspace} % as in f \from X \to Y
\newcommand{\SigmaMod}{\Sigma_{\textup{mod}}}

\newtheorem{theorem}{Theorem}[section]
\newtheorem*{theorem*}{Theorem}
\newtheorem{proposition}[theorem]{Proposition}
\newtheorem{lemma}[theorem]{Lemma}
\newtheorem{corollary}[theorem]{Corollary} 
\newtheorem{algorithm}[theorem]{Algorithm} 

\newtheorem*{namedtheorem}{\theoremname}
\newcommand{\theoremname}{testing}
\newenvironment{named}[1]{\renewcommand{\theoremname}{#1}\begin{namedtheorem}}{\end{namedtheorem}}

\theoremstyle{definition}
\newtheorem{definition}[theorem]{Definition}
\newtheorem{example}[theorem]{Example}

\newtheorem{remark}[theorem]{Remark}

\newcommand{\refthm}[1]{Theorem~\ref{Thm:#1}}
\newcommand{\reflem}[1]{Lemma~\ref{Lem:#1}}
\newcommand{\refprop}[1]{Proposition~\ref{Prop:#1}}
\newcommand{\refcor}[1]{Corollary~\ref{Cor:#1}}

\newcommand{\refdef}[1]{Definition~\ref{Def:#1}}

\newcommand{\reffig}[1]{Figure~\ref{Fig:#1}}

\title[A lower bound for the volumes of modular link complements]{A lower bound for the volumes of modular link complements} 

\author[Connie Hui]{Connie On Yu Hui} 
\address[]{School of Mathematics, Monash University, Australia } 
\email[]{\rm onyu.hui@monash.edu} 

\author{Dionne Ibarra}
\address{School of Mathematics, Monash University, Australia}
\email{{\rm dionne.ibarra@monash.edu}}

\author{Jos\'{e} Andr\'{e}s Rodr\'{\i}guez-Migueles}
\address[]{Ludwig Maximilian University of Munich, Germany}
\email[] {\rm migueles@math.lmu.de}

\begin{document} 

\begin{abstract}
Every finite collection of oriented closed geodesics in the modular surface has a canonically associated link in its unit tangent bundle coming from the periodic orbits of the geodesic flow.  We study the volume of the associated link complement with respect to its unique complete hyperbolic metric. We provide the first lower volume bound that is linear in terms of the number of distinct exponents in the code words corresponding to the collection of closed geodesics.  
\end{abstract} 

\maketitle

\section{Introduction} 
Let $\Sigma$ be a complete, orientable hyperbolic surface or $2$-orbifold of finite area. A finite collection $\Gamma$  of oriented closed geodesics in $\Sigma$ has a canonical lift $\widehat\Gamma$ in its unit tangent bundle $\mathrm{T^1}\Sigma$, namely, the corresponding collection of periodic orbits of the geodesic flow. Let $\mathrm{T^1}\Sigma\setminus{\widehat\Gamma}$ denote the complement of a closed regular neighbourhood of $\widehat\Gamma$ in $\mathrm{T^1}\Sigma$.  As a consequence of the Hyperbolization Theorem, $\mathrm{T^1}\Sigma\setminus{\widehat\Gamma}$ admits a complete hyperbolic metric of finite volume  if and only if $\Gamma$ fills $\Sigma$ \cite{Foulon-Hasselblatt:ContactAnosov}. By the Mostow-Prasad Rigidity Theorem, such a metric is unique up to isometry, meaning that any geometric invariant is a topological invariant.  Recently, there has been interest in expressing the volume of $\mathrm{T^1}\Sigma\setminus{\widehat\Gamma}$ in terms of quantities related to the properties of $\Gamma$ (for example, see \cite{Bergeron-Pinsky-Silberman:UpperBound, RodriguezMigueles:LowerBound, RodriguezMigueles:PeriodsOfContinuedFractions, Hui-Rodriguez:BunchesUpperVolume}). 
\vskip .2cm

In this paper, we mainly focus on the case of the \emph{modular surface} $\Sigma_{\textup{mod}}$,  which is the quotient of $\HH^2$ by the modular group $\mathrm{PSL}(2,\ZZ)$.  This hyperbolic $2$-orbifold is particularly interesting because its unit tangent bundle is homeomorphic to the complement of the trefoil knot in $\mathbb{S}^3$. Therefore, in this particular case, $\mathrm{T^1}\Sigma_{\textup{mod}}\setminus\widehat \Gamma$ can be considered as a link complement in $\mathbb{S}^3$. The canonical lift $\widehat \Gamma$ is called a \emph{modular link} and the $3$-manifold $\mathrm{T^1}\Sigma_{\textup{mod}}\setminus\widehat \Gamma$ is called a \emph{modular link complement}. Moreover, it was observed by Ghys in \cite{Ghys:KnotsAndDynamics} 
 that after trivially Dehn filling the trefoil cusp of $\mathrm{T^1}\Sigma_{\textup{mod}}$, the periodic orbits of the geodesic flow over the modular surface are the Lorenz links in $\mathbb{S}^3.$ 
\vskip .2cm
A closed geodesic in the modular surface $\Sigma_{\textup{mod}}$ is called a \emph{modular geodesic}.
Modular geodesics are in one-to-one correspondence with conjugacy classes of hyperbolic elements in $\mathrm{PSL}(2,\ZZ)$, i.e.\ those with trace more than two (see \cite{Brandts-Pinsky-Silberman:VolModularLinks}, Section 3).
Let $\footnotesize {L:=\pm\begin{pmatrix} 
1 & 1 \\
0 & 1 
\end{pmatrix}}$ 
and $\small {R:=\pm\begin{pmatrix} 
1 & 0 \\
1 & 1 
\end{pmatrix}}$ be representatives of elements in $\mathrm{PSL}(2,\ZZ) \subset \mathrm{PSL}(2,\RR) \cong \mathrm{Isom}_+(\HH^2)$.  These are parabolic elements with traces equal to two. Notice that
any word of $L$'s and $R$'s with the presence of both letters is a matrix element that represents a hyperbolic element in $\mathrm{PSL}(2,\ZZ)$. Moreover, the conjugacy classes of hyperbolic primitive elements in $\mathrm{PSL}(2,\ZZ)$ are in one-to-one correspondence with primitive words of $L$'s and $R$'s (with positive powers) up to cyclic permutations. Such primitive words are said to be the \emph{code words} of the corresponding modular geodesics and can be stated in the form $\prod_{i=1}^{n} L^{a_i} R^{b_i}$, which starts with a letter $L$, ends with a letter $R$, we call each of $a_1, \ldots, a_n$ an \emph{$L$-exponent} of the word, and we call each of $b_1, \ldots, b_n$ an \emph{$R$-exponent} of the code word. 
\vskip .2cm
Bergeron, Pinsky and Silberman in \cite{Bergeron-Pinsky-Silberman:UpperBound} studied the problem of finding an upper bound for the volumes of modular link complements.  They expressed the upper volume bound in terms of the symbolic descriptions of the modular geodesics, namely, the word periods and word exponents. Recently, in \cite{Hui-Rodriguez:BunchesUpperVolume}, the first and the third authors provided an upper volume bound that is independent of word exponents and is quadratic in the word periods. 
\vskip .2cm
In \cite{RodriguezMigueles:LowerBound}, the third author gave a lower bound for the volumes of canonical lift complements of geodesics in hyperbolic surfaces admitting a non-trivial pants decomposition; this lower bound does not hold for the modular surface. Nevertheless, by a finite cover argument, one can construct sequences of modular knots, where the corresponding sequence of volumes of their complements are bounded below linearly in terms of the word period \cite[Theorem~1.4]{RodriguezMigueles:LowerBound}.  However, there is a sequence of modular knot complements with word periods growing to infinity and yet the volumes are bounded above by a constant \cite[Corollary~1.2]{RodriguezMigueles:LowerBound}. This suggests that a general lower volume bound for all modular link complements cannot grow to infinity as word periods grow to infinity. 
\vskip .2cm

In this paper, we give a lower volume bound that is linear in the number of distinct word exponents for all modular link complements. 

\begin{named}{\refcor{LowerBound}} 
Let $\widehat\Gamma$ be a modular link.  If $A$ is the set of all distinct $L$-exponents and $B$ is the set of all distinct $R$-exponents in the code words for $\widehat\Gamma$, then we have
 \[\mathrm{Vol}(\mathrm{T^1}\Sigma_{\textup{mod}} \setminus \widehat\Gamma)\geq \frac{v_{\textup{tet}}}{72} (|A|+|B|-10),\]
where $v_{\textup{tet}}\approx 1.01494$ is the volume of the regular ideal tetrahedron.  
\end{named} 

\refcor{LowerBound} is a consequence of \refthm{LowerBound1}, which gives a comparatively larger lower bound that is less direct to compute. Such a lower bound is expressed in terms of the set cardinalities of a \emph{maximal $L$-exponent set} and a \emph{maximal $R$-exponent set} of a modular link, the definitions of which can be found in \refdef{AtildeBtilde}. 

\begin{named}{\refthm{LowerBound1}} 
Let $\widehat{\Gamma}$ be a modular link. Let $\widetilde{A}$ and $\widetilde{B}$ be a maximal $L$-exponent set and a maximal $R$-exponent set of the modular link $\widehat\Gamma$ respectively. 
    We have  
    \[\mathrm{Vol}(\mathrm{T^1}\Sigma_{\textup{mod}} \setminus \widehat{\Gamma})\geq \frac{v_{\textup{tet}}}{12}(|\widetilde{A} |+|\widetilde{B}|),\]
    where $v_{\textup{tet}}\approx 1.01494$ is the hyperbolic volume of the regular ideal tetrahedron. 
\end{named} 

A lower bound analogous to \refcor{LowerBound} for infinite families of closed geodesics in the thrice-punctured sphere was proved by the third author in \cite[Theorem 1.6]{RodriguezMigueles:PeriodsOfContinuedFractions}. The main profound result for this argument uses a result from Agol, Storm and Thurston~\cite{Agol-Storm-Thurston} by finding and cutting the corresponding canonical lift complement along a particular incompressible surface. In particular, this result was used to prove \cite[Theorem 1.5]{RodriguezMigueles:LowerBound}. 
\vskip .2cm

\refthm{LowerBound1} mainly uses the understanding of the relationship between subwords in a code word and their topological meanings in an ideally triangulated $6$-fold once-punctured torus cover of the modular surface. Another main tool we use to prove \refthm{LowerBound1} is the once-punctured torus case of the aforementioned lower bound result in \cite[Theorem 1.5]{RodriguezMigueles:LowerBound}.
\vskip .2cm

The upper bound can be found in \cite[Theorem 6.2]{Hui-Rodriguez:BunchesUpperVolume} whose proof involves the combinatorial method of constructing an ideal triangulation of the manifold \cite[Subsection~6.5]{Thurston:Geom&TopOf3Mfd}.  By putting together this result with \refthm{LowerBound1}, we have that up to some multiplicative and additive constants, the lower volume bound in \refthm{LowerBound1} is sharp for an infinite family of modular knot complements. 

\begin{named}{\refcor{LinearBound}}
Let $n, i\in\ZZ_+$.  Let $\widehat\gamma$ be a modular knot with code word $L^{a_1} R^i  L^{a_2} R^i \ldots L^{a_n} R^i$, where $n$ is the word period of $\widehat\gamma$.  If \[a_1 > a_2 > \ldots > a_{n-1} > a_n \quad \text{and \quad for each } j\in\ZZ\cap[2,n], \ a_{j-1}-a_j \geq 6,  \] then 
 \[\frac{v_{\textup{tet} }}{12}n\leq \mathrm{Vol}(\mathrm{T^1}\Sigma_{\textup{mod}}\setminus\widehat\gamma) \leq 8 v_{\textup{tet}} (7n+2),\]
where $v_{\textup{tet}}\approx 1.01494$ is the volume of the regular ideal tetrahedron. 
\end{named} 
\vskip .2cm

Cremaschi, Krifka, Mart\'inez-Granado, and Vargas Pallete provided in \cite{CKMP:VolBdCanonicalLiftComplOfRandomGeodesic} a lower volume bound for the complements of canonical lifts of filling geodesics in closed hyperbolic surfaces. Other recent related works include \cite{CRY:VolAndFillingCollectionsOfMulticurves}, which use pants graph to provide volume bounds.  In this article, we provide a lower volume bound for all modular link complements that is easy to compute, such a lower bound can be obtained directly by reading the corresponding code words that define the links.

\subsection{Organisation} 
In Section~\ref{prelim} we give preliminary details on canonical lifts of smooth essential closed curves in a hyperbolic surface to its unit tangent bundle. In Section~\ref{Sec:Coding} we define an $E$-triangulated once-punctured torus, $LR$-code words and cutting sequences, and then we define 
 linear and winding subwords. In Section~\ref{Sec:LowerBounds} we present our main results. In Section \ref{Sec:Classification} we give a classification of modular link complements. 

\subsection{Acknowledgements} 
The authors thank the organizers and participants of the workshop Dynamics, Foliation, and Geometry III at the mathematical research institute MATRIX in Australia, where this collaboration began. The second author thanks the support by the Australian Research Council grant DP210103136 and the third author thanks the Special Priority Programme SPP 2026 Geometry at Infinity funded by the DFG.

\section{Preliminaries}\label{prelim}
Throughout this paper, we denote the number of elements in a set $S$ by~$|S|$.

\subsection{Canonical lift complement}

Let $ \Sigma=\Sigma_{g,k}$ be a \textit{hyperbolic surface}, which is an orientable smooth surface of genus $g$ with $k$ punctures and negative Euler characteristic. Consider the $3$-manifold $\mathrm{T^1}\Sigma$, which is the unit tangent bundle associated to $\Sigma$. For any finite collection of  smooth essential closed curves $\Gamma$ in $\Sigma$, there is a \emph{canonical} lift $\widehat \Gamma$ in $\mathrm{T^1}\Sigma$ realised by the set of tangent vectors to $\Gamma$. By drilling the canonical lift $\widehat\Gamma$ from $\mathrm{T^1}\Sigma$, we obtain a $3$-manifold $\mathrm{T^1}\Sigma\setminus{\widehat\Gamma}$. Foulon and Hasselblatt in \cite{Foulon-Hasselblatt:ContactAnosov} observed that when $\Gamma$ is filling and in a minimal position, then $\mathrm{T^1}\Sigma\setminus{\widehat\Gamma}$  admits a complete hyperbolic metric of finite volume. In particular, $\mathrm{Vol}(\mathrm{T^1}\Sigma\setminus{\widehat\Gamma} )$ is a topological invariant.
%is such an invariant.

 \vskip .2cm
Care needs to be taken when dealing with homotopy classes of closed curves and canonical lifts because a homotopy between two curves does not always induce an isotopy between the lifts of the curves in the unit tangent bundle. For example, introducing self-intersections may change the isotopy class of the lift.  % \marginconnie{If we have time, can I learn more about this paragraph?}

 \begin{definition}
     A homotopy $h:\mathbb S^1\times I\rightarrow \Sigma$ between two self-transverse closed curves is \emph{transversal} if $h_t$ is self-transverse for all $t\in I$.
 \end{definition} 
 
 A transversal homotopy $h_t$ between closed curves $\alpha$ and $\beta$ in $\Sigma$ induces an isotopy in $\mathrm{T^1}\Sigma$ between $\widehat\alpha$ and $\widehat\beta$, see \cite[Sec. 2]{RodriguezMigueles:LowerBound}. Hass-Scott in \cite{HassScott:Minimalposition} proved that any two self-transverse minimal representatives are transversally homotopic. In particular, they are transversally homotopic to a unique closed geodesic for a given metric with negative curvature on $\Sigma$. Thus, from \cite[Sec. 2.1]{RodriguezMigueles:LowerBound} we have:
 
 \begin{lemma}
     Suppose $\Gamma_1$ and $\Gamma_2$ are two collections of oriented essential closed curves in $\Sigma$ %and  each element in $\Gamma_i$ representing a unique free homotopy class, 
     such that $\Gamma_1\cup\Gamma_2$ is in a minimal position and there exists a homotopy between $\Gamma_1$ and $\Gamma_2$. Then, $\mathrm{T^1}\Sigma\setminus{\widehat\Gamma_1} \cong \mathrm{T^1}\Sigma\setminus{\widehat\Gamma_2} $.
 \end{lemma}\label{lem:trans}

Therefore, when $\Gamma$ is a collection of essential closed curves in a minimal position, it is enough to look at the geodesic representative of $\Gamma$ in a hyperbolic metric on $\Sigma$ or even in a singular flat metric on $\Sigma$, provided that the closed curves avoid the singularities. 

\subsection{Homotopy classes of arcs} \ 

Throughout this subsection, we define $C=(\SS^1\times [0,1])\setminus \{(1,\frac{1}{2})\}$ to be a once-punctured annulus. 

\begin{definition}\label{Homotopyarcs}
   Given the punctured annulus $C$, we say that two oriented arcs $\alpha,\beta:[0,1] \rightarrow C$ with $\alpha(\{0,1\})\cup \beta(\{0,1\})\subset \partial C$ are in the \emph{same homotopy class in $C$} if there exists a homotopy $h:{[0,1]_1\times[0,1]_2}\rightarrow{C}$ such that for any $s_2\in [0,1]_2$,  $$h_0(s_2)=\alpha(s_2),  \hspace{.2cm} h_1(s_2)=\beta(s_2) \hspace{.2cm} \mbox{and} \hspace{.2cm}h([0,1]_1\times\{0,1\})\subset \partial C.$$
\end{definition} 
\begin{remark}\label{ArcsCodingCylinder}
The arcs in Definition~\ref{Homotopyarcs} can be encoded by using two \textit{oriented seam simple arcs}, which are arcs that start at different components of $\partial C$ and end at the puncture of $C$. The coding comes from assigning a letter in $\{a,a^{-1},b,b^{-1}\}$ to each oriented intersection of the given arc with the seam arcs as shown in Figure \ref{1arcs}. Hence, we can associate each properly embedded arc in $C$ with a code word. We assume every code word to be in reduced form. 
\end{remark}

\begin{definition}\label{Def:LastWindingNumber}
      Let $\alpha$ be an oriented arc in $C$ in minimum position. We say that $k \in \ZZ$ is the \textit{last winding number of $\alpha$ relative to the puncture} if $(b^{-1}a)^k$ is the last subword in the code word for $\alpha$ that contains the word $b^{-1}a$ and that $(b^{-1}a)^k$ has the largest subword length. 
\end{definition} 

\begin{figure}[ht]
%% Creator: Inkscape 1.1.2 (b8e25be833, 2022-02-05), www.inkscape.org
%% PDF/EPS/PS + LaTeX output extension by Johan Engelen, 2010
%% Accompanies image file 'abDirections.pdf' (pdf, eps, ps)
%%
%% To include the image in your LaTeX document, write
%%   \input{<filename>.pdf_tex}
%%  instead of
%%   \includegraphics{<filename>.pdf}
%% To scale the image, write
%%   \def\svgwidth{<desired width>}
%%   \input{<filename>.pdf_tex}
%%  instead of
%%   \includegraphics[width=<desired width>]{<filename>.pdf}
%%
%% Images with a different path to the parent latex file can
%% be accessed with the `import' package (which may need to be
%% installed) using
%%   \usepackage{import}
%% in the preamble, and then including the image with
%%   \import{<path to file>}{<filename>.pdf_tex}
%% Alternatively, one can specify
%%   \graphicspath{{<path to file>/}}
%% 
%% For more information, please see info/svg-inkscape on CTAN:
%%   http://tug.ctan.org/tex-archive/info/svg-inkscape
%%
\begingroup%
  \makeatletter%
  \providecommand\color[2][]{%
    \errmessage{(Inkscape) Color is used for the text in Inkscape, but the package 'color.sty' is not loaded}%
    \renewcommand\color[2][]{}%
  }%
  \providecommand\transparent[1]{%
    \errmessage{(Inkscape) Transparency is used (non-zero) for the text in Inkscape, but the package 'transparent.sty' is not loaded}%
    \renewcommand\transparent[1]{}%
  }%
  \providecommand\rotatebox[2]{#2}%
  \newcommand*\fsize{\dimexpr\f@size pt\relax}%
  \newcommand*\lineheight[1]{\fontsize{\fsize}{#1\fsize}\selectfont}%
  \ifx\svgwidth\undefined%
    \setlength{\unitlength}{109.79073727bp}%
    \ifx\svgscale\undefined%
      \relax%
    \else%
      \setlength{\unitlength}{\unitlength * \real{\svgscale}}%
    \fi%
  \else%
    \setlength{\unitlength}{\svgwidth}%
  \fi%
  \global\let\svgwidth\undefined%
  \global\let\svgscale\undefined%
  \makeatother%
  \begin{picture}(1,1.18566719)%
    \lineheight{1}%
    \setlength\tabcolsep{0pt}%
    \put(0,0){\includegraphics[width=\unitlength,page=1]{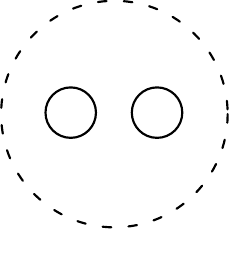}}%
    \put(0.42338669,0.02015555){\makebox(0,0)[lt]{\lineheight{1.25}\smash{\begin{tabular}[t]{l}(a)\end{tabular}}}}%
    \put(0,0){\includegraphics[width=\unitlength,page=2]{abDirections.pdf}}%
    \put(0.11824834,0.84616682){\color[rgb]{1,0,0}\makebox(0,0)[lt]{\lineheight{1.25}\smash{\begin{tabular}[t]{l}$a^{-1}$\end{tabular}}}}%
    \put(0.11824834,0.47728289){\color[rgb]{1,0,0}\makebox(0,0)[lt]{\lineheight{1.25}\smash{\begin{tabular}[t]{l}$a$\end{tabular}}}}%
    \put(0.78986589,0.84616683){\color[rgb]{1,0,0}\makebox(0,0)[lt]{\lineheight{1.25}\smash{\begin{tabular}[t]{l}$b^{-1}$\end{tabular}}}}%
    \put(0.81719056,0.47728289){\color[rgb]{1,0,0}\makebox(0,0)[lt]{\lineheight{1.25}\smash{\begin{tabular}[t]{l}$b$\end{tabular}}}}%
  \end{picture}%
\endgroup%
 
\quad\quad\quad
%% Creator: Inkscape 1.1.2 (b8e25be833, 2022-02-05), www.inkscape.org
%% PDF/EPS/PS + LaTeX output extension by Johan Engelen, 2010
%% Accompanies image file 'bab-1a-1.pdf' (pdf, eps, ps)
%%
%% To include the image in your LaTeX document, write
%%   \input{<filename>.pdf_tex}
%%  instead of
%%   \includegraphics{<filename>.pdf}
%% To scale the image, write
%%   \def\svgwidth{<desired width>}
%%   \input{<filename>.pdf_tex}
%%  instead of
%%   \includegraphics[width=<desired width>]{<filename>.pdf}
%%
%% Images with a different path to the parent latex file can
%% be accessed with the `import' package (which may need to be
%% installed) using
%%   \usepackage{import}
%% in the preamble, and then including the image with
%%   \import{<path to file>}{<filename>.pdf_tex}
%% Alternatively, one can specify
%%   \graphicspath{{<path to file>/}}
%% 
%% For more information, please see info/svg-inkscape on CTAN:
%%   http://tug.ctan.org/tex-archive/info/svg-inkscape
%%
\begingroup%
  \makeatletter%
  \providecommand\color[2][]{%
    \errmessage{(Inkscape) Color is used for the text in Inkscape, but the package 'color.sty' is not loaded}%
    \renewcommand\color[2][]{}%
  }%
  \providecommand\transparent[1]{%
    \errmessage{(Inkscape) Transparency is used (non-zero) for the text in Inkscape, but the package 'transparent.sty' is not loaded}%
    \renewcommand\transparent[1]{}%
  }%
  \providecommand\rotatebox[2]{#2}%
  \newcommand*\fsize{\dimexpr\f@size pt\relax}%
  \newcommand*\lineheight[1]{\fontsize{\fsize}{#1\fsize}\selectfont}%
  \ifx\svgwidth\undefined%
    \setlength{\unitlength}{109.79073727bp}%
    \ifx\svgscale\undefined%
      \relax%
    \else%
      \setlength{\unitlength}{\unitlength * \real{\svgscale}}%
    \fi%
  \else%
    \setlength{\unitlength}{\svgwidth}%
  \fi%
  \global\let\svgwidth\undefined%
  \global\let\svgscale\undefined%
  \makeatother%
  \begin{picture}(1,1.18566719)%
    \lineheight{1}%
    \setlength\tabcolsep{0pt}%
    \put(0,0){\includegraphics[width=\unitlength,page=1]{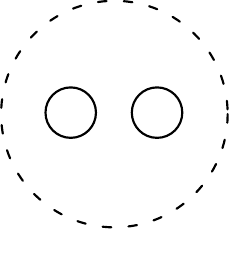}}%
    \put(0.42338669,0.02015555){\makebox(0,0)[lt]{\lineheight{1.25}\smash{\begin{tabular}[t]{l}(b)\end{tabular}}}}%
    \put(0,0){\includegraphics[width=\unitlength,page=2]{bab-1a-1.pdf}}%
  \end{picture}%
\endgroup%

\caption{(a) Letters assigned to the orientations of arcs (transversely) intersecting the seam edges in the punctured annulus. (b) A properly embedded arc corresponding to the code word $bab^{-1}a^{-1}$.   
\label{1arcs}} 
\end{figure} 

\begin{lemma}\label{Lem:LastWindingNumber}
  Let $\alpha$, $\beta$ be two oriented arcs in the once-punctured annulus $C$. If the last winding numbers of $\alpha$ and $\beta$ relative to the puncture are different, then the homotopy classes of the arcs $\alpha$ and $\beta$ are different. 
\end{lemma} 
\vskip .2cm

\begin{proof}We will show the contrapositive of the statement. Suppose the oriented arcs $\alpha$ and $\beta$ are in the same homotopy class. 

Note that for any $ab$-code word, its corresponding arc is homotopic to an arc with an $ab$-code word that has minimum word length; this arc is constructed by unwinding the arc around each of the two components of $\partial C$. We may thus assume $\alpha$ and $\beta$ to be the arcs corresponding to the two $ab$-code words with minimal word lengths. 

Glue the two boundary components of $\partial C$ in a natural way such that each of $\alpha$ and $\beta$ becomes a loop, with base point $p$, that does not have meridional windings near the glued meridional circle of the resultant once-punctured torus $T\setminus\{(1,\frac{1}{2})\}$. Denote the loops embedded in the once-punctured torus by ${\alpha'}$ and ${\beta'}$ respectively. 

Since $\alpha$ and $\beta$ are in the same homotopy class of arcs  and there are no meridional windings after the gluing, the oriented loops ${\alpha'}$ and ${\beta'}$ are in the same homotopy class in $\pi_1(T\setminus\{(1,\frac{1}{2})\},p)\cong \ZZ\ast\ZZ \eqqcolon \langle u,v \rangle$.

Note that the last winding number $k_{\alpha}$ of the arc $\alpha$ relative to the puncture corresponds to the subword $(uvu^{-1}v^{-1})^{k_{\alpha}}$ or $(uvu^{-1}v^{-1})^{-k_{\alpha}}$ (last subword of this type with maximal power) in $\ZZ\ast\ZZ$. Similarly,  the last winding number $k_{\beta}$ of the arc $\beta$ relative to the puncture corresponds to the subword $(uvu^{-1}v^{-1})^{k_{\beta}}$ or $(uvu^{-1}v^{-1})^{-k_{\beta}}$ in $\ZZ\ast\ZZ$. Since ${\alpha'}$ and ${\beta'}$ are in the same homotopy class in $\pi_1(T\setminus\{(1,\frac{1}{2})\},p)\cong \ZZ\ast\ZZ \eqqcolon \langle u,v \rangle$, they can be represented by the same reduced word in $u$ and $v$. Since every element in a free group $\langle u,v\rangle$ corresponds to a unique reduced word, we have $k_{\alpha} = k_{\beta}$. 
\end{proof}
\vskip .2cm

\section{Coding of curves in a minimal position}%\label{Coding}
\label{Sec:Coding}

\subsection{Code words for modular geodesics}\label{Modulargeodesicscode} 
The quotient of $\HH^2$ by the modular group $\mathrm{PSL}(2,\ZZ)$, called the \emph{modular surface} and denoted by $\SigmaMod$, is an orbifold that is a sphere with a cusp, a cone point of order three, and a cone point of order two. A fundamental domain for $\SigmaMod$ is given by one third of the $0, 1, \infty$ ideal triangle.
\vskip .2cm

Background information on the modular group can be found in many places, for example in the work of Series~\cite{Series:ModularSurface}; also see Brandts, Pinsky, and Silberman~\cite{Bergeron-Pinsky-Silberman:UpperBound}. We review a few relevant facts below. 
\vskip .2cm

Consider the upper half plane $\HH^2$ with its hyperbolic metric. Let $U$ be a rotation of $\pi$ about the point $i$ and let $V$ be a rotation of $2\pi/3$ about the point $\frac{1}{2}+i\frac{\sqrt{3}}{2}$, permuting points $\infty, 1, 0$.  These two rotations generate the modular group $\mathrm{PSL}(2,\ZZ)$.
As elements of $\mathrm{PSL}(2,\ZZ)$, $U$ and $V$ have the form
 \[ U = \pm \begin{pmatrix} 0 & -1 \\ 1 & 0 \end{pmatrix} \quad 
 V = \pm \begin{pmatrix} 0 & -1 \\ 1 & -1 \end{pmatrix} \]
The rotation $V$ fixes the hyperbolic ideal triangle in $\HH^2$ with vertices $0, 1, \infty$, while $U$ maps it to an adjacent ideal triangle. Thus the orbit of this ideal triangle under $\mathrm{PSL}(2,\ZZ)$ is an invariant tessellation of $\HH^2$ by ideal triangles called the Farey tessellation. It has an ideal vertex at each point of $\mathbb{Q}\cup\infty$ on $\partial\HH^2$. 
\vskip .2cm

Elements of finite order in $\mathrm{PSL}(2,\ZZ)$ are exactly the conjugates of $1, U, V, V^2$.
Up to conjugation, every element of infinite order can be written in positive powers of
$L=V^2U$ and $R=VU$.

 \vskip .2cm
A closed geodesic in the modular surface $\Sigma_{\textup{mod}}$ is called a \emph{modular geodesic}.
Modular geodesics are in one-to-one correspondence with conjugacy classes of hyperbolic elements in $\mathrm{PSL}(2,\ZZ)$, i.e.\ those with trace more than two.

 \vskip .2cm
Note that $R$ and $L$ are parabolic elements, with trace two, but
any word in positive powers in $R$ and $L$ involving both letters is hyperbolic.

\vskip .2cm
A modular geodesic lifts to $\HH^2$, tiled by the Farey tessellation. Series observed that such lifts cut out a sequence of triangles~\cite{Series:ModularSurface}. Within a given triangle an oriented geodesic enters through one side and then either exits through the side on its left (cutting off a single ideal vertex on its left side) or exits to its right. The sequence of rights and lefts determines a word in positive powers of $R$ and $L$ up to cyclic order called the \emph{cutting sequence}. This agrees with the matrix product corresponding to the geodesic.

\vskip .2cm
The term \emph{code word (for the modular geodesic)} in this paper has the same meaning as that in \cite{Hui-Rodriguez:BunchesUpperVolume} or \emph{Lorenz word} in \cite{Birman-Williams:KnottedPeriodicOrbitsInDynamicalSystemsI}. In this paper, we assume all code words for modular geodesics to be in a standard form $\prod_{i=1}^{n} L^{a_i} R^{b_i} = L^{a_1} R^{b_1}L^{a_2} R^{b_2} \ldots L^{a_n} R^{b_n}$, which starts with a letter $L$, ends with a letter $R$, and is not a power of any proper subwords. Note that the bases $L$ and $R$ are alternating in the standard form of the code word. We call each of $a_1, \ldots, a_n$ an \emph{$L$-exponent} of the word, and we call each of $b_1, \ldots, b_n$ an \emph{$R$-exponent} of the code word. 

\vskip .2cm
An \emph{unlabelled $L$-exponent} is an $L$-exponent viewed as a positive integer while a \emph{labelled $L$-exponent} is an $L$-exponent viewed as a labelled positive integer. An \emph{unlabelled $R$-exponent} and a \emph{labelled $R$-exponent} are defined similarly. For example, the code word ${L}^{10}{R}^2{L}^{10}{R}^2{L}^{10}{R}^6$ has three distinct labelled $L$-exponents and three distinct labelled $R$-exponents; and it has only one unlabelled $L$-exponent and two distinct unlabelled $R$-exponents.

\subsection{Once-punctured torus cover and its $E$-triangulation}  
First consider a closed torus with no punctures, denoted by $\Sigma_{1,0}$. That is, the surface of genus one with zero punctures. Fix a choice of generators $\mu' = 1/0$ and $\lambda'=0/1$ for $\pi_1(\Sigma_{1,0})$, then any free homotopy class of a simple closed curve in the torus is determined by a slope of the form $p\mu' + q \lambda' $ which can be written as $p/q$; an element of $\QQ\cup\{1/0\}$. In particular, an unoriented geodesic representative of $p/q$ is a simple closed curve that has a constant tangent line; the curve lifts to a line of constant slope $p/q$ in the universal cover $\RR^2$.  

\vskip .2cm
Now consider a once-punctured torus, denoted by $\Sigma_{1,1}$; the genus one surface with one puncture.
%Now consider the once-punctured torus, which we denote by $\Sigma_{1,1}$: the genus one surface with one puncture. 

\begin{definition}\label{Abeliancover} Consider the universal cover $\RR^2$ of a torus. We define the \emph{abelian cover of the once-punctured torus} $\Sigma_{1,1}$ as the covering space corresponding to the commutator subgroup $[\pi_1(\Sigma_{1,1}),\pi_1(\Sigma_{1,1})]$ of $\pi_1(\Sigma_{1,1})\cong\ZZ\ast\ZZ$. 
\end{definition} 

The abelian cover of the punctured torus; can be viewed as the plane $\RR^2$ with integer lattice points removed (see \reffig{AbelianCoverWithEquilTri}). The line $x=0$ in $\RR^2$ projects to an arc $\mu$ in $\Sigma_{1,1}$ with both endpoints on the puncture. Similarly, the line $y=0$ projects to an arc $\lambda$. Consider those simple closed curves in the punctured torus that are parallel to lines in $\RR^2$ of rational slope $p/q$, but disjoint from points in the integer lattice. These lines of rational slope project to closed curves in $\Sigma_{1,1}$ meeting $\mu$ a total of $q$ times, and meeting $\lambda$ a total of $p$ times. We let $p/q$ denote this simple closed curve. In particular, a simple closed curve parallel to $\mu$ is $1/0$, and one parallel to $\lambda$ is $0/1$. Note these are not all the closed curves in $\Sigma_{1,1}$; for example we are omitting curves that wrap around the puncture in more complicated ways.

\begin{definition} \label{torusFareytessellation}
  We define the \emph{$E$-triangulation of the once-punctured torus} $\Sigma_{1,1}$ as the ideal triangulation of $\Sigma_{1,1}$ obtained by adding an ideal edge $\lambda$ parallel to $0/1$,  an ideal edge $\mu$ parallel to $1/0$, and an ideal arc parallel to the slope $1/1$ (see \reffig{ETriangulation}). This divides $\Sigma_{1,1}$ into two ideal triangles, which we view as equilateral triangles in this paper. 
\end{definition} 

\begin{figure}[ht]
%% Creator: Inkscape 1.1.2 (b8e25be833, 2022-02-05), www.inkscape.org
%% PDF/EPS/PS + LaTeX output extension by Johan Engelen, 2010
%% Accompanies image file 'ETriangulation.pdf' (pdf, eps, ps)
%%
%% To include the image in your LaTeX document, write
%%   \input{<filename>.pdf_tex}
%%  instead of
%%   \includegraphics{<filename>.pdf}
%% To scale the image, write
%%   \def\svgwidth{<desired width>}
%%   \input{<filename>.pdf_tex}
%%  instead of
%%   \includegraphics[width=<desired width>]{<filename>.pdf}
%%
%% Images with a different path to the parent latex file can
%% be accessed with the `import' package (which may need to be
%% installed) using
%%   \usepackage{import}
%% in the preamble, and then including the image with
%%   \import{<path to file>}{<filename>.pdf_tex}
%% Alternatively, one can specify
%%   \graphicspath{{<path to file>/}}
%% 
%% For more information, please see info/svg-inkscape on CTAN:
%%   http://tug.ctan.org/tex-archive/info/svg-inkscape
%%
\begingroup%
  \makeatletter%
  \providecommand\color[2][]{%
    \errmessage{(Inkscape) Color is used for the text in Inkscape, but the package 'color.sty' is not loaded}%
    \renewcommand\color[2][]{}%
  }%
  \providecommand\transparent[1]{%
    \errmessage{(Inkscape) Transparency is used (non-zero) for the text in Inkscape, but the package 'transparent.sty' is not loaded}%
    \renewcommand\transparent[1]{}%
  }%
  \providecommand\rotatebox[2]{#2}%
  \newcommand*\fsize{\dimexpr\f@size pt\relax}%
  \newcommand*\lineheight[1]{\fontsize{\fsize}{#1\fsize}\selectfont}%
  \ifx\svgwidth\undefined%
    \setlength{\unitlength}{197.81974288bp}%
    \ifx\svgscale\undefined%
      \relax%
    \else%
      \setlength{\unitlength}{\unitlength * \real{\svgscale}}%
    \fi%
  \else%
    \setlength{\unitlength}{\svgwidth}%
  \fi%
  \global\let\svgwidth\undefined%
  \global\let\svgscale\undefined%
  \makeatother%
  \begin{picture}(1,0.67914931)%
    \lineheight{1}%
    \setlength\tabcolsep{0pt}%
    \put(0,0){\includegraphics[width=\unitlength,page=1]{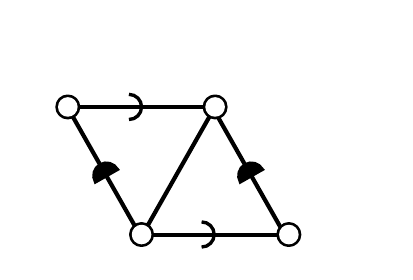}}%
    \put(0.48210862,0.00980169){\makebox(0,0)[lt]{\lineheight{1.25}\smash{\begin{tabular}[t]{l}$\lambda$\end{tabular}}}}%
    \put(0.14858798,0.2384761){\makebox(0,0)[lt]{\lineheight{1.25}\smash{\begin{tabular}[t]{l}$\mu$\end{tabular}}}}%
    \put(0,0){\includegraphics[width=\unitlength,page=2]{ETriangulation.pdf}}%
    \put(0.90024544,0.09160286){\color[rgb]{0.50196078,0.50196078,0}\makebox(0,0)[lt]{\lineheight{1.25}\smash{\begin{tabular}[t]{l}$x$\end{tabular}}}}%
    \put(-0.00420848,0.63635857){\color[rgb]{0.50196078,0.50196078,0}\makebox(0,0)[lt]{\lineheight{1.25}\smash{\begin{tabular}[t]{l}$y$\end{tabular}}}}%
  \end{picture}%
\endgroup%
 
\caption{$E$-triangulation of the once-punctured torus $\Sigma_{1,1}$. 
\label{Fig:ETriangulation}} 
\end{figure}

A union of curves is said to be in a \emph{minimal position relative to the $E$-triangulated once-punctured torus} $\Sigma_{1,1}$ if the union of curves intersect the edges of the triangulation and themselves transversely and there are no monogons or bigons when we consider the union of curves and triangulation edges as a graph (see Theorem 2 in \cite{HassScott:monobigono}).  
\vskip .2cm 

From now on, unless otherwise specified, an arc or subarc in the $E$-triangulated once-punctured torus refers to an arc with endpoints lying in some edge(s) of the $E$-triangulation. Also, a curve (which is a closed loop or an arc) in the $E$-triangulated once-punctured torus is assumed to be in a minimal position relative to the $E$-triangulated once-punctured torus. 
\vskip .2cm

The abelian cover of $\Sigma_{1,1}$ can then be viewed as a tiling of $\RR^2$ by these equilateral triangles, where all vertices of the equilateral triangles are removed. 

\begin{lemma}[Lemma 2.2 in \cite{Pinsky-Purcell-Rodriguez:ArithmeticModularLinks}]\label{Lem:6FoldCover}
The once-punctured torus forms a 6-fold cover of the modular surface. The group of deck transformations is generated by a rotation of order three and a rotation of order two. 
\end{lemma}

A Farey tessellation of the hyperbolic plane $\HH^2$ involves an ideal triangulation of the upper half plane.  Each of the ideal triangles is a $3$-fold cover of the modular surface with ideal vertices lying in the boundary $\partial\HH^2 \cong \RR\cup\{\infty\}$. If we consider two such neighbouring ideal triangles, they form a fundamental region in $\HH^2$ for the once-punctured torus. In this paper, instead of considering the Farey tessellation in the covering space $\HH^2$, we consider the abelian cover $\RR^2\setminus\ZZ^2$ where each ideal triangle has ideal vertices located at points with integral coordinates in $\RR^2$. Notice that the projection to the once-punctured torus gives an ideal triangulation which is isotopic to the one in Definition~\ref{torusFareytessellation}. 

\begin{figure}[ht]
%% Creator: Inkscape 1.1.2 (b8e25be833, 2022-02-05), www.inkscape.org
%% PDF/EPS/PS + LaTeX output extension by Johan Engelen, 2010
%% Accompanies image file 'AbelianCoverWithEquilTri.pdf' (pdf, eps, ps)
%%
%% To include the image in your LaTeX document, write
%%   \input{<filename>.pdf_tex}
%%  instead of
%%   \includegraphics{<filename>.pdf}
%% To scale the image, write
%%   \def\svgwidth{<desired width>}
%%   \input{<filename>.pdf_tex}
%%  instead of
%%   \includegraphics[width=<desired width>]{<filename>.pdf}
%%
%% Images with a different path to the parent latex file can
%% be accessed with the `import' package (which may need to be
%% installed) using
%%   \usepackage{import}
%% in the preamble, and then including the image with
%%   \import{<path to file>}{<filename>.pdf_tex}
%% Alternatively, one can specify
%%   \graphicspath{{<path to file>/}}
%% 
%% For more information, please see info/svg-inkscape on CTAN:
%%   http://tug.ctan.org/tex-archive/info/svg-inkscape
%%
\begingroup%
  \makeatletter%
  \providecommand\color[2][]{%
    \errmessage{(Inkscape) Color is used for the text in Inkscape, but the package 'color.sty' is not loaded}%
    \renewcommand\color[2][]{}%
  }%
  \providecommand\transparent[1]{%
    \errmessage{(Inkscape) Transparency is used (non-zero) for the text in Inkscape, but the package 'transparent.sty' is not loaded}%
    \renewcommand\transparent[1]{}%
  }%
  \providecommand\rotatebox[2]{#2}%
  \newcommand*\fsize{\dimexpr\f@size pt\relax}%
  \newcommand*\lineheight[1]{\fontsize{\fsize}{#1\fsize}\selectfont}%
  \ifx\svgwidth\undefined%
    \setlength{\unitlength}{217.74212549bp}%
    \ifx\svgscale\undefined%
      \relax%
    \else%
      \setlength{\unitlength}{\unitlength * \real{\svgscale}}%
    \fi%
  \else%
    \setlength{\unitlength}{\svgwidth}%
  \fi%
  \global\let\svgwidth\undefined%
  \global\let\svgscale\undefined%
  \makeatother%
  \begin{picture}(1,0.90582246)%
    \lineheight{1}%
    \setlength\tabcolsep{0pt}%
    \put(0,0){\includegraphics[width=\unitlength,page=1]{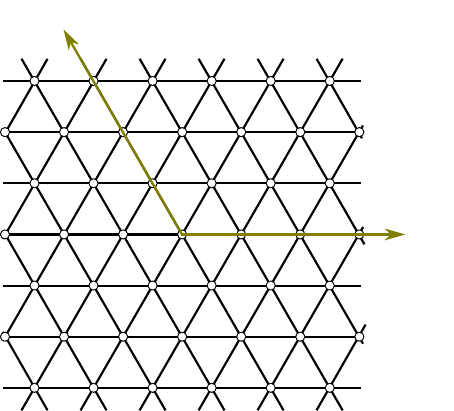}}%
    \put(0.90937251,0.37203376){\color[rgb]{0.50196078,0.50196078,0}\makebox(0,0)[lt]{\lineheight{1.25}\smash{\begin{tabular}[t]{l}$x$\end{tabular}}}}%
    \put(0.08767188,0.86694688){\color[rgb]{0.50196078,0.50196078,0}\makebox(0,0)[lt]{\lineheight{1.25}\smash{\begin{tabular}[t]{l}$y$\end{tabular}}}}%
  \end{picture}%
\endgroup%
 
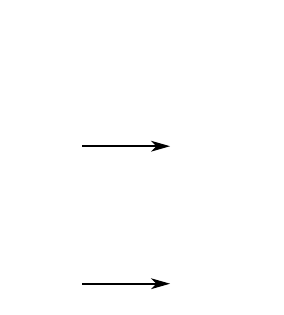
\caption{Left: The lifted $E$-triangulation in the abelian cover $\RR^2\setminus\ZZ^2$ of the once-punctured torus. Each ideal triangle (with sides glued appropriately) is a $3$-fold cover of the modular surface. Right: Covering maps $p$, $\widehat{p}$, $q$, and bundle projection maps $\pi$, $\widehat\pi$. 
\label{Fig:AbelianCoverWithEquilTri}} 
\end{figure}

\begin{definition}\label{E-Homotopyarcs}
   Given the $E$-triangulated once-punctured torus. We say that two oriented arcs $\alpha,\beta:[0,1] \rightarrow \Sigma_{1,1},$ with $\alpha(\{0,1\})\cup \beta(\{0,1\})$ contained in the ideal $E$-edges of the $E$-triangulation, are in the \emph{same $E$-homotopy class in $\Sigma_{1,1}$} if and only if there exists a homotopy $h:{[0,1]_1\times[0,1]_2}\rightarrow{\Sigma_{1,1}}$ such that for any $s_2\in[0,1]_2$,
   $$h_0(s_2)=\alpha(s_2),  \hspace{.2cm} h_1(s_2)=\beta(s_2) \hspace{.2cm} $$ 
   and $h([0,1]_1\times\{0,1\})$ is contained in $E$-edges. The homotopy $h$ is called an $E$-homotopy. 
\end{definition}

\subsection{Cutting sequences of curves in a minimal position in the once-punctured torus}In this section we generalise Section 6 in \cite{Pinsky-Purcell-Rodriguez:ArithmeticModularLinks} to 
obtain a coding for any closed curves in a minimal position in the $E$-triangulated once-punctured torus.

\begin{definition}\label{LR-CutTorus}
Let $\tau$ be an oriented, connected curve (not necessarily closed) in a minimal position relative to the $E$-triangulated once-punctured torus.  If $\tau$ is not closed, we assume the endpoints of $\tau$ always lie in the edges of the $E$-triangulation. 

The \emph{$LR$-cutting sequence} of $\tau$ is a sequence of letters $L$'s and $R$'s obtained as follows: Consider an intersection point between $\tau$ and an edge of the triangulation. The oriented curve $\tau$ either runs to the edge on the left or the edge on the right. If it runs to the left, take the letter $L$ as the first term in the sequence. If it runs to the right, take the letter $R$ instead. Now repeat a similar process for the next intersection point to obtain the second term of the sequence. By repeating a similar process for each intersection point, we obtain the $LR$-cutting sequence of $\tau$. 

When $\tau$ is a closed curve, the corresponding $LR$-cutting sequence is periodic. By abuse of notations, we usually consider the maximal non-repeating part of such cutting sequence as the cyclic permutation class of the cutting sequence.
\end{definition} 
\vskip .2cm

Notice that for a single lift, in the $E$-triangulated once-punctured torus, of a given closed geodesic in the modular surface, the $LR$-cutting sequence of the lift agrees with the matrix product corresponding to the modular geodesic described in Subsection~\ref{Modulargeodesicscode}.
\vskip .2cm

If we ignore the diagonal of the fundamental region of the $E$-triangulation, we can similarly define another type of cutting sequence, which we call the $XY$-cutting sequence, to give us the notion of linear subwords. Details will be given in the next subsection.

\subsection{Linear subwords and winding subwords} 

Given $\omega$ an $LR$-code word corresponding to a curve in a minimal position in the $E$-triangulated once-punctured torus, we wish to characterise the topological behaviour in terms of its $LR$-cutting sequence. We start by providing an explicit description of the $XY$-cutting sequence of an unoriented simple closed geodesics in $\Sigma_{1,1}.$

\vskip .2cm
We will first consider a \emph{parallelogram grid} $\Lambda$ with its horizontal sides labelled $X$ and its other sides labelled $Y$ as in Figure~\ref{fig:XYcuttingsequenceEtrang}. Let $\beta$ be any straight line in the plane.
The line $\beta$ meets the sides, $X$ and $Y$, in such a way that while walking along $\beta$ we obtain a certain sequence in terms of $X$ and $Y$ (with the convention that $\beta$ avoids all vertices of the square grid), we call this the \textit{$XY$-cutting sequence} of $\beta$. Suppose $\beta$ has positive rational slope $p/q$. Series in \cite{SeriesMarkoff} (also see \cite{ChristoffelObserv, Smithcontinuedfractions}) observed that if $p/q >1$, then in the $XY$-cutting sequence of $\beta$, the appearances of $Y$ are isolated. That is, between any two $Y$'s, there is at least one $X$ between them.  She further observed that between any two $Y$'s there are either $\lfloor p/q \rfloor$ or $\lceil p/q \rceil$ $X$'s between them. We call the collection of all consecutive $X$'s an \emph{X-block}. Also, if $p/q<1$, then the roles of $X$ and $Y$ are reversed and between any two $X$'s there are either $\lfloor q/p \rfloor$ or $\lceil q/p \rceil$ $Y$'s between them. We call the collection of all consecutive $Y$'s a \emph{Y-block}.

\begin{theorem}[\cite{SeriesMarkoff}, Theorem 5.5 in \cite{Davis:Dynamics}]\label{Thm:admissible_XYwords}
    The $XY$-cutting sequence of an unoriented simple closed geodesic in $\Sigma_{1,1}$ is a primitive $XY$-code word (up to cyclic permutation) that satisfies at least one of the following two statements: 
\begin{enumerate}
    \item \label{Item:IsolatedY} The appearances of $Y$ are isolated (meaning between any two $Y$'s, there is at least one $X$) and the $X$-blocks have length $k$ or $k+1$ with $k\in \mathbb{N}$ depending on the slope of the geodesic. 
    \item \label{Item:IsolatedX} The appearances of $X$ are isolated (meaning between any two $X$'s, there is at least one $Y$) and the $Y$-blocks have length $k$ or $k+1$ with $k\in \mathbb{N}$ depending on the slope of the geodesic.
\end{enumerate}
\end{theorem}

We can obtain the $LR$-cutting sequence from the $XY$-cutting sequence of an oriented geodesic in $\Sigma_{1,1}$ by following an extension of Algorithm 6.4 in \cite{Pinsky-Purcell-Rodriguez:ArithmeticModularLinks}, as stated below:

\begin{algorithm}\label{Alg:substitution}
Let $(p,q)\in\ZZ^2\setminus\{(0,0)\}$ and $\omega$ be the $XY$-code word (of length $n$) corresponding to the slope $p/q$. 
\begin{enumerate}
  
  \item If $p$ and $q$ are both positive or both negative, then we apply the following steps to create an $LR$-code word: Starting with an empty word $\omega'$, for an integer $j$ starting from $1$ to $n$, 
\begin{enumerate}
\item If the $j$-th letter of $\omega$ is $Y$ and the next letter is $X$, append $L$ to $\omega'$.
\item If the $j$-th letter of $\omega$ is $X$ followed by $Y$, append  $R$ to $\omega'$.
\item If the $j$-th letter of $\omega$ is $Y$ followed by $Y$, append  $RL$ to $\omega'$.
\item If the $j$-th letter of $\omega$ is $X$ followed by $X$, append  $LR$ to $\omega'$.
\end{enumerate}

  \item If $p$ and $q$ have opposite signs, then we apply the following steps to create an $LR$-code word: For an integer $j$ starting from $1$ to $n$, 
\begin{enumerate}

\item If the $j$-th letter of $\omega$ is $Y$ and the next letter is $X$, append $R^2$ to $\omega'$.
\item If the $j$-th letter of $\omega$ is $X$ followed by $Y$, append $L^2$ to $\omega'$.
\item If the $j$-th letter of $\omega$ is $Y$ followed by $Y$, append $RL$ to $\omega'$.
\item If the $j$-th letter of $\omega$ is $X$ followed by $X$, append $LR$ to $\omega'$.

\end{enumerate}

\item If the oriented direction is $(0,\pm 1)$ or $(\pm1,0)$, then the $LR$-cutting sequence corresponding to $\omega$ is $LR$.
\end{enumerate}

\end{algorithm}

\begin{example} \label{Ex:LinearLRWord}
    By applying Algorithm~\ref{Alg:substitution} on the $XY$-code word \[Y(XX)Y(XX)Y(X)Y(XX)\] (see Figure~\ref{fig:XYcuttingsequenceEtrang}), the $LR$-cutting sequence obtained is \[(L)(LR)(R)(L)(LR)(R)(L)(R)(L)(LR)(R).\] 
\end{example}

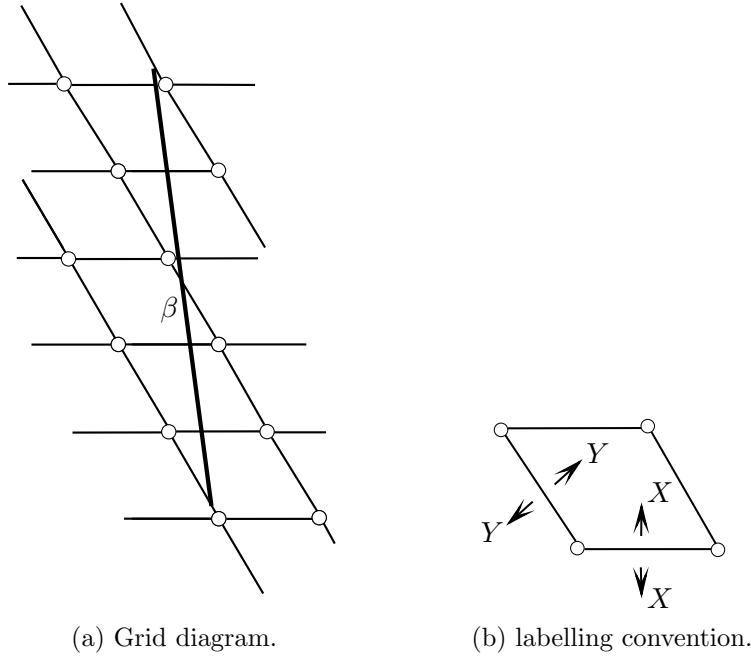
\begin{figure}[ht]
\centering
\begin{subfigure}{.5\textwidth}
\centering
$\vcenter{\hbox{\begin{overpic}[scale = .95]{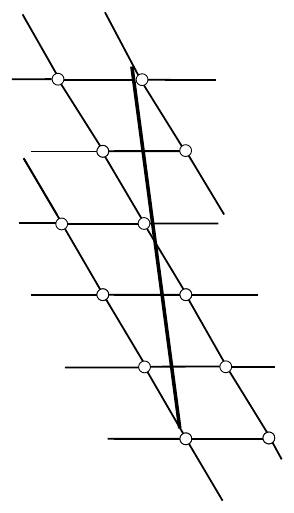}
\put(63,110){$\beta$}
\end{overpic} }}$
\caption{Grid diagram.} \label{fig:griddiagram}
\end{subfigure}
\centering
\begin{subfigure}{.4\textwidth}
\centering
$\vcenter{\hbox{\begin{overpic}[scale = .9]{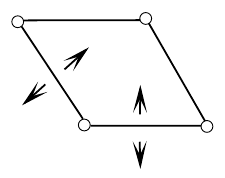}
\put(40, 55){$Y$}
\put(0, 25){$Y$}
\put(63, 40){$X$}
\put(63, 0){$X$}
\end{overpic} }}$
\caption{labelling convention.} \label{fig:orientedgriddiagramET}
\end{subfigure}
\caption{(a) The line $\beta$ with rational slope $5/2$ has $XY$-cutting sequence $Y(XX)Y(XXX)$. (b) The unoriented $XY$-cutting sequence is obtained by assigning the letter $X, Y$ to the two possible intersection types.} \label{fig:XYcuttingsequenceEtrang}
\end{figure}

\begin{definition}\label{Def:Linearword}
  We say that an $LR$-code word is  \emph{linear} if and only if it is a subword of a code word obtained by Algorithm~\ref{Alg:substitution} from an $XY$-code word described in Theorem \ref{Thm:admissible_XYwords}. A \emph{linear subword} $\omega_l$ of an $LR$-code word $\omega$ is a subword of $\omega$ such that $\omega_l$ is linear. 
\end{definition} 

    In other words, an $LR$-code word $\omega$ representing a curve $\tau$ (not necessarily closed) in $\Sigma_{1,1}$ is linear if and only if there exists a simple closed geodesic $\overline\tau$ in $\Sigma_{1,1}$ such that 
    \begin{itemize}
        \item the $XY$-code word $\chi$ of $\overline\tau$ satisfies Statement~(\ref{Item:IsolatedY}) or Statement~(\ref{Item:IsolatedX}) in \refthm{admissible_XYwords}; and
        \item the $LR$-code word $\omega$ is a subword of the $LR$-code word obtained from the $XY$-code word $\chi$ via Algorithm~\ref{Alg:substitution}. 
    \end{itemize} 
\vskip .2cm

For computability purposes, one can apply Algorithm 7.3 in \cite{Davis:Dynamics} to obtain the continued fraction expansion of the slope of a linear $LR$-code word relative to the corresponding $XY$-code word. 
\vskip .2cm

\begin{definition}\label{Windingword}
    We say that an $LR$-code word is \emph{$k$-winding clockwise} if it is of the form $LR^{6k+r}L$ where  $k\in\ZZ_+$, and $r\in [0,5] \cap\ZZ.$ Similarly, we say that an $LR$-code word is \emph{$k$-winding counter-clockwise} if it is of the form $RL^{6k+r}R$ where  $k\in\ZZ_+$, and $r\in [0,5]\cap\ZZ.$ Moreover, given a $k$-winding (counter)clockwise $LR$-code word, we call the exponent $(6k+r)$ a \emph{central exponent}. A \emph{winding subword} $\omega_w$ of an $LR$-code word $\omega$ is a subword of $\omega$ such that $\omega_w$ is $k$-winding clockwise or $k$-winding counter-clockwise for some positive integer $k$. 
\end{definition}

\begin{lemma}\label{Lem:kTimesAroundCusp}
   Let $\tau$ be an oriented curve (possibly an arc) in a minimal position relative to the $E$-triangulated once-punctured torus as in Definition~\ref{LR-CutTorus}. If the  $LR$-code word of $\tau$ is $k$-winding  (counter)clockwise, then $\tau$ winds (counter)clockwise $k$ times around the cusp of $\Sigma_{1,1}$.
\end{lemma}

\begin{figure}[ht]
    \centering
    \includegraphics[scale=1.5]{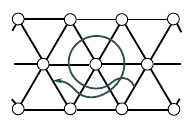}
    \caption{An illustration of an oriented arc winding once clockwise corresponding to the $LR$-code word $LR^7L$.}
    \label{fig:kwinding}
\end{figure}
\begin{proof}
Consider the oriented curve $\tau$ corresponding to the $LR$-code word that is $k$-winding clockwise. From the definition, this code word is of the form $LR^{6k+r}L$ where $k \in \ZZ_+$, $r \in \ZZ$ and $|r|<6$. Notice that the subword $R^6$ corresponds to the segment of $\tau$ that winds once around the cusp of $\Sigma_{1,1}$ in the clockwise direction as illustrated in Figure~\ref{fig:kwinding}. Since $|r|<6$, then $\tau$ winds clockwise $k$ times around the cusp of $\Sigma_{1,1}$. A similar argument can be made for the case of $RL^{6k+r}R$. (See \reffig{WindArc1} for $k=1$ and $r\in\{0,1,2,3,4,5\}$.)
\end{proof} 
\vskip .2cm

%\begin{definition}\label{Def:FirstlinearnextWindingword}
%Given an $LR$-code word $\omega$ with a winding subword $\omega_i = LR^{6k+r}L$ (or $RL^{6k+r}R$ resp.) such that $\omega_i$ is not contained in the words $LR^{6k+r}LR^s$ and $LR^{6k+r}L^s$ for $s \geq 3$ (or $RL^{6k+r}RL^s$ and $RL^{6k+r}R^s$ resp.), we say that $\omega_{i,l}$ is the \emph{first linear subword of $\omega$ starting at $\omega_i$} if and only if $\omega_{i,l}$ is a linear subword of $\omega$ such that both of the following hold: 

%\begin{enumerate}
%        \item The first letter of $\omega_{i,l}$ is a letter in the subword $R^{6k+r}L$ (or $L^{6k+r}R$ resp.) of $\omega_i$. 
%        \item If $\overline{\omega_{i,l}}$ is a subword of $\omega$ such that $\overline{\omega_{i,l}}$ contains $\omega_{i,l}$ and the word length of $\overline{\omega_{i,l}}$ is larger than $\omega_{i,l}$ by exactly one, then any arc representing $\overline{\omega_{i,l}}$ has a self-intersection up to $E$-homotopy. 
%    \end{enumerate}
   \begin{definition}\label{Def:FirstlinearnextWindingword}
%\textbf{New definition}
% \marginandres{I think adding the red correction then both definitions are the same. \dionne{Why do we need to add the winding piece to $\tau$?. It seems unnecessary. $\alpha$ will be realised by the linear word only, it doesn't need to come from the intersection points. When I was playing around with intersection points I found that for winding pieces you can move the intersection points into many different segments of the word so it's too ambiguous. Its more concrete to just take the linear word. }}
 Given an $LR$-code word $\omega$ with a winding subword $\omega_i = LR^{6k+r}L$ (or $RL^{6k+r}R$) we say that $\omega_{i,l}$ is the \emph{first linear subword} of $\omega$ starting at $\omega_i$ if it is the largest subword for which all the following conditions hold:
 \vskip .2cm

     \begin{enumerate}
        \item The first letter of $\omega_{i,l}$ is a letter in the subword $R^{6k+r}L$ (or $L^{6k+r}R$ resp.) of $\omega_i$. 
        \item $\omega_{i,l}$ is a linear subword (notice that the arc realised by $\omega_{i,l}$ contains no self-intersections).
        
            \item  If $\overline{\omega_{i,l}}$ is a subword of $\omega$ such that $\overline{\omega_{i,l}}$ contains $\omega_{i,l}$ as the initial subword and contains the consecutive  monomial relative to $R$ and $L$, then any arc representing $\overline{\omega_{i,l}}$ has a self-intersection up to $E$-homotopy. 
            \end{enumerate}

    \end{definition}
    %%%%%%%%%%%%%%%%%%%%%%%%%%%%%%%%%%%%%%%%
    \begin{example}
        Consider the subword $\omega' = L^2R^2$ of the word $\omega=RL^6R^6L^{12}RL$. The consecutive monomial relative to $R,L$ of $\omega'$ is $R^4L$.
    \end{example}
    
\vskip .2cm
\begin{example}
    Consider an $LR$-code word  $$RL^{12}RL^2R^2L^2RLRL.$$ The first linear subword starting at $RL^{12}R$ is
    $RL^2R^2L^2RLR$. See \reffig{firstlinearsubword}.
    
    Notice that this word is linear (has $XY$-code word $XYXXXYX$) and satisfies the three conditions in Definition~\ref{Def:FirstlinearnextWindingword}.  

\begin{figure}[ht]
    \centering
    \begin{subfigure}{.39\textwidth}
    \centering
         \includegraphics[scale=0.9]{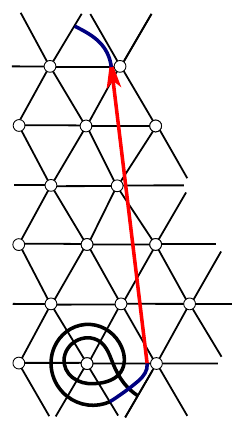}
         \caption{A representation of the $LR$-code word $RL^{12}RL^2R^2L^2RLRL$ in the abelian cover.}
    \end{subfigure} \quad
    \begin{subfigure}{.39\textwidth}
    \centering
         \includegraphics[scale=0.45]{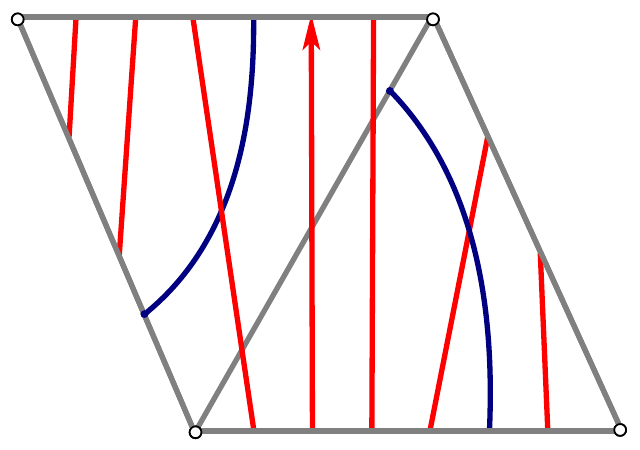}
         \caption{A representation of the $LR$-code word $LRL^2R^2L^2RLRL$ in 
         $\Sigma_{1,1}$.}
    \end{subfigure}
    \caption{An illustration of the first linear subword of the $LR$-code word $RL^{12}RL^2R^2L^2RLRL$.} 
    \label{Fig:firstlinearsubword} 
\end{figure} %\marginandres{maybe would consistent to have the unit squares as the gluing of to equilateral triangles fundamental domain.}
    
% The following explains the maximality\sting:

%\begin{itemize}
    %\item If we add $L$ on the right, then, on one side, either the next letter to it is $L$ which induces a non-removable bigon relative to a vertical straight line, which is a contradiction because two straight lines cannot intersect more than once in the Euclidean plane. On the other side, if the next letter to it is $R,$ implies that we get an $XY$-code word of type $XYXXYXXX,$ then there are $XXXX$ and $XX$ as $X$-blocks in a $XY$-code word described in Theorem \ref{Thm:admissible_XYwords}. Contradiction. 
    %\item If we add $L$ on the left, then, on one side, either the previous letter to it is $L.$ This means we would have $L^3$ as a subword, which is impossible for a linear $LR$-code word. On the other side, if the next letter to it is $R,$ implies that we get an $XY$-code word of type $XYXYXXYXX,$ then there are $XXX,$ $X$ and $XX$ as $X$-blocks in a $XY$-code word described in Theorem \ref{Thm:admissible_XYwords}. Contradiction. 
%\end{itemize}    

\end{example}

%\begin{lemma} \label{Lem:IntersectArc}
%    Let $\tau$ be an oriented arc in the $E$-triangulated once-punctured torus $\Sigma_{1,1}$ corresponding to an $LR$-subword $\overline{\omega_{i,l}}$ defined in Definition \ref{Def:FirstlinearnextWindingword}. Then all but one homotopy class of essential non-peripheral simple closed curve in $\Sigma_{1,1}$ intersects $\tau$ up to $E$-homotopy. 
%\end{lemma} 

%\begin{proof}
 %   By Condition~$(2)$ in Definition~\ref{Def:FirstlinearnextWindingword}, we have that $\tau$ has a self-intersection up to $E$-homotopy which implies that it contains an essential non-peripheral simple closed loop $\alpha$ based at a self-intersection point and is a sub-arc of $\tau$. If we take $\beta$ to be an essential non-peripheral simple closed curve whose homotopy class is different from that of $\alpha$, then $\beta$ intersects $\tau$ up to $E$-homotopy.
%\end{proof} 

\begin{lemma} \label{Lem:IntersectArc}
Let $\omega$ be an $LR$-code word with a winding subword $\omega_i = LR^{6k+r}L$ (or $RL^{6k+r}R$) and let $\tau$ be an oriented arc in the $E$-triangulated once-punctured torus $\Sigma_{1,1}$ corresponding to the $LR$ code word starting at $\omega_i$ and ending in $\overline{\omega_{i,l}}$ defined in Condition (3) of Definition \ref{Def:FirstlinearnextWindingword}. Then all homotopy classes of essential non-peripheral simple closed curves different from one in $\Sigma_{1,1}$ intersect $\tau$ up to $E$-homotopy. 
\end{lemma}
\begin{proof}
     By Condition~$(3)$ in Definition~\ref{Def:FirstlinearnextWindingword}, we have that $\tau$ has a self-intersection up to $E$-homotopy and there exists an essential non-peripheral simple closed loop $\alpha$ (see \reffig{alphafirstlinearsubword} for an example) with $LR$-code word a subword of  $\omega_{i,l}$, as described in Definition~\ref{Def:FirstlinearnextWindingword}, that also realises a sub-arc of $\tau$. If we take $\beta$ to be an essential non-peripheral simple closed curve whose homotopy class is different from that of $\alpha$, then $\beta$ intersects $\tau$, up to $E$-homotopy.
\end{proof}

\begin{figure}[ht]
    \centering
    $\vcenter{\hbox{\begin{overpic}[scale = .6]{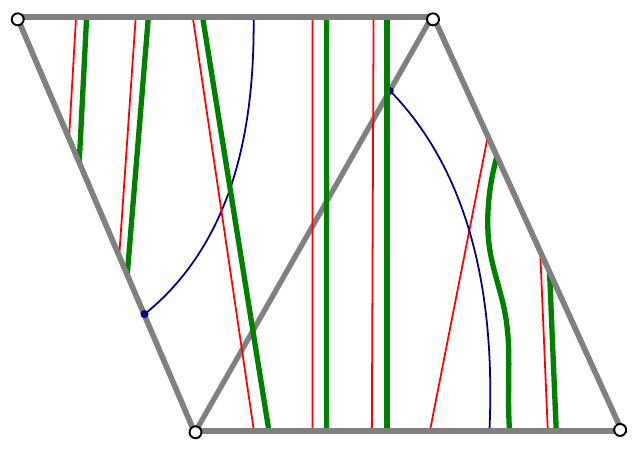}
\put(43, 70){$\alpha$}
\end{overpic} }}$  
    \caption{An illustration of an essential non-peripheral simple closed loop, $\alpha$, realised by the first linear subword of $RL^{12}RL^2R^2L^2RLRL$.} 
    \label{Fig:alphafirstlinearsubword} 
\end{figure} 

\section{Lower bounds for modular link complements} \label{Sec:LowerBounds} 

One of the main tools to prove the main results of this section (\refthm{LowerBound1} and \refcor{LowerBound}) is given by the following particular instance of Theorem~1.5 in \cite{RodriguezMigueles:LowerBound}, as restated below: 

\begin{theorem}\label{Thm:HmtpClassesLowerBd}\cite[Theorem 1.5]{RodriguezMigueles:LowerBound}
Let $\beta$ be an essential simple closed curve in a once-punctured torus $\Sigma_{1,1}$ and $T$ be a union of closed curves that fill $\Sigma_{1,1}$. If $T\cup\beta$ is in a minimal position, then we have 
 $$\mathrm{Vol}(\mathrm{T^1}\Sigma_{1,1}\setminus \widehat{T})\geq \frac{v_{\textup{tet}}}{2}\sharp\{\mbox{non-simple homotopy classes of} \hspace{.2cm}  T\mbox{-arcs in} \hspace{.2cm} \Sigma_{1,1}\setminus\beta \},$$ 
 where $v_{\textup{tet}}\approx 1.01494$ is the volume of the regular ideal tetrahedron.
 \end{theorem} 
\vskip .2cm

In this section, unless otherwise specified, define the symbol $\widehat{\Gamma}$ to be a modular link with $c$ $LR$-code words \[\omega^{(1)} \coloneqq \prod_{i=1}^{n_1} L^{a_i} R^{b_i},\quad \omega^{(2)} \coloneqq \prod_{i=n_1+1}^{n_1+n_2} L^{a_i} R^{b_i},\ldots, \quad \omega^{(c)} \coloneqq \prod_{i= n_1+\ldots + n_{c-1}+1}^{n_1+n_2+\ldots + n_c} L^{a_i} R^{b_i}, \] with $\Gamma$ denoting the union of the modular geodesics $\pi(\widehat\Gamma)$. Denote by $\Omega$ the finite collection $\{\omega^{(1)}, \omega^{(2)}, \ldots, \omega^{(c)}\}$. Let $A$ be the set of all distinct $L$-exponent(s) for $\Omega$ and let $B$ be the set of all distinct $R$-exponent(s) for $\Omega$. 
\vskip .2cm

To state and prove \refprop{Homotopyclasses} more succinctly, we define the following: 

\begin{definition} \label{Def:AtildeBtilde}
Denote by $\widetilde{A}$ a maximal (in terms of set cardinality) set of labelled $L$-exponent(s) for $\Omega$ 
 such that 
 \begin{itemize}
     \item ($\forall a_i\in \widetilde{A}, a_i\geq 6$) and 
     \item ($\forall a_i, a_j \in  \widetilde{A} \textup{ with } i\neq j, |a_i-a_j|\geq 6$).
 \end{itemize}
We call  $\widetilde{A}$ a \emph{maximal $L$-exponent set of the modular link $\widehat\Gamma$}. 

A \emph{maximal $R$-exponent set $\widetilde{B}$ of the modular link~$\widehat\Gamma$} is defined similarly.
\end{definition} 

\begin{figure}[ht]
			\centering 
 \includegraphics[scale=.9] {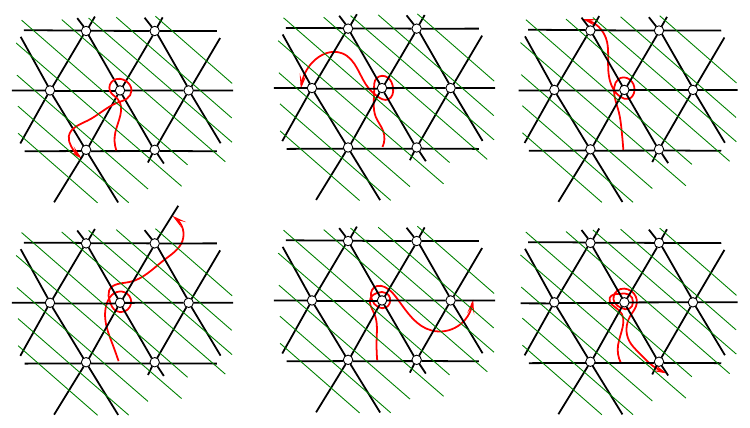}
				\caption{The $E$-triangulated abelian cover $\mathbb{R}^2\setminus \ZZ^2$ (black);  some lifts (green) of the  simple closed geodesic $\beta$ corresponding to the $LR$-code word $LRL^2R^2$; and the lifts (red) of the six arcs $\alpha$'s, the subword of each of which is $LR^{6+r}L^3$ for some $r\in\ZZ\cap[0,5]$.}
    \label{Fig:WindArc1}
		\end{figure}

To prove Theorem~\ref{Thm:LowerBound1}, we use the following key proposition: 
\begin{proposition}\label{Prop:Homotopyclasses}
     Let $\widetilde{A}$ and $\widetilde{B}$ be a maximal $L$-exponent set and a maximal $R$-exponent set of the modular link $\widehat\Gamma$ respectively. Suppose $T = p^{-1}(\Gamma)$ is in a minimal position relative to the $E$-triangulated once-punctured torus $\Sigma_{1,1}$. There exists a simple closed curve $\beta$ in $\Sigma_{1,1}$ such that the union of curves in $T$ can be homotoped to a minimal position with 
     \[\sharp\{\mbox{non-simple homotopy classes of} \hspace{.2cm}  T\mbox{-arcs in} \hspace{.2cm} \Sigma_{1,1}\setminus\beta\}\geq |\widetilde{A} |+|\widetilde{B}|.\]
\end{proposition} 

\begin{proof}  
Observe that the set $\widetilde{A}\cup\widetilde{B}$ has $(|\widetilde{A}|+|\widetilde{B}|)$ elements, which   correspond to $(|\widetilde{A}|+|\widetilde{B}|)$ winding subwords of the $LR$-code word $\omega$ of $\widehat\Gamma$. If $\widetilde{A}\cup\widetilde{B}=\emptyset$, then $|\widetilde{A}|+|\widetilde{B}| = 0$ and the inequality holds. Now suppose $\widetilde{A}\cup\widetilde{B}\neq\emptyset$. Denote by $\omega_1$, $\omega_2$, \ldots, $\omega_{|\widetilde{A}|+|\widetilde{B}|}$ the winding subwords aforementioned. 
\vskip .2cm

For each $i\in\ZZ\cap[1,|\widetilde{A}|+|\widetilde{B}|]$, consider the winding piece, $\omega_i$, 
let $\tau_i$ be an oriented arc in the $E$-triangulated once-punctured torus $\Sigma_{1,1}$ corresponding to the $LR$ code word starting at $\omega_i$ and ending in $\overline{\omega_{i,l}}$ defined in Condition (3) of Definition \ref{Def:FirstlinearnextWindingword}.  
By \reflem{IntersectArc}, any homotopy class of essential non-peripheral simple closed curves in $\Sigma_{1,1}$ other than one homotopy class, denoted by $[\alpha_i]$, intersects the first linear piece of $\tau_i$ up to $E$-homotopy. 
%\end{enumerate}

\vskip .2cm

Take $\beta$ to be an essential non-peripheral simple closed curve that is not in any one of the homotopy classes in $\{ [\alpha_1], \cdots, [\alpha_{|\widetilde{A}|+|\widetilde{B}|}] \}$.
%$\{ [\alpha_j] \}_{j \in \mathcal{I}} \cup \{ [\alpha^{[0]}], [\alpha^{[1]}], [\alpha^{[-1]}], [\alpha^{[\infty]}]\}$ (notice that $|\mathcal{I}| \leq |\widetilde{A}|+|\widetilde{B}|$). 
It follows that $\beta$ intersects the linear arc that follows each arc corresponding to the winding subwords $\omega_1, \omega_2, \ldots, \omega_{|\widetilde{A}|+|\widetilde{B}|}$ up to $E$-homotopy. 

\vskip .2cm

Hence, $\beta$ chops $T$ into oriented arcs (with orientation induced from that of $T$) in the once-punctured annulus $\Sigma_{1,1}\setminus \beta$. For each $i\in\ZZ\cap[1,|\widetilde{A}|+|\widetilde{B}|]$, there exists a unique chopped oriented arc that contains the winding arc with winding subword $\omega_i$. By \reflem{kTimesAroundCusp} and the definition of $\widetilde{A}\cup\widetilde{B}$, each such chopped oriented arc has a unique last winding number (see \refdef{LastWindingNumber}).   As there are ($|\widetilde{A}|+|\widetilde{B}|$) distinct last winding numbers, by \reflem{LastWindingNumber}, there are at least ($|\widetilde{A}|+|\widetilde{B}|$) distinct homotopy classes of $T$-arcs in $\Sigma_{1,1}\setminus\beta$. 
\end{proof}
\vskip .2cm

\begin{theorem}  \label{Thm:LowerBound1}  
    Let $\widehat{\Gamma}$ be a modular link. Let $\widetilde{A}$ and $\widetilde{B}$ be a maximal $L$-exponent set and a maximal $R$-exponent set of the modular link $\widehat\Gamma$ respectively. 
    We have  
    \[\mathrm{Vol}(\mathrm{T^1}\Sigma_{\textup{mod}} \setminus \widehat{\Gamma})\geq \frac{v_{\textup{tet}}}{12}(|\widetilde{A} |+|\widetilde{B}|),\]
    where $v_{\textup{tet}}\approx 1.01494$ is the hyperbolic volume of the regular ideal tetrahedron. 
\end{theorem} 

\begin{proof} 
Denote by $\Gamma$ the projection of $\widehat\Gamma$ in the modular surface $\Sigma_{\textrm{mod}}$.  

By \reflem{6FoldCover}, there exists a $6$-fold covering map $p \from \Sigma_{1,1}\to \Sigma_{\textrm{mod}}$. We then consider the preimage of $\Gamma$ via $p$.  As volume is multiplicative with respect to finite coverings, we have
    \[6\mathrm{Vol}(\mathrm{T^1}\Sigma_{\textup{mod}} \setminus \widehat\Gamma)=\mathrm{Vol}(\mathrm{T^1}\Sigma_{1,1} \setminus \widehat{p}\hspace{2pt}^{-1}(\widehat\Gamma)),\]where $\widehat{p}$ is the  horizontal covering map $\widehat{p}\from \mathrm{T^1}\Sigma_{1,1}\to \mathrm{T^1}\Sigma_{\textup{mod}}$ induced by $p$. (See \reffig{AbelianCoverWithEquilTri}, right.) 

    By Theorem~\ref{Thm:HmtpClassesLowerBd} and Proposition~\ref{Prop:Homotopyclasses}, there exists a simple closed curve $\beta$ in $\Sigma_{1,1}$ such that the union of curves in $p^{-1}(\Gamma)$ can be homotoped to a minimal position with  
 \begin{align*}
     &\mathrm{Vol}(\mathrm{T^1}\Sigma_{1,1} \setminus \widehat{p}\hspace{2pt}^{-1}(\widehat{\Gamma})) \\
     &\geq \frac{v_{\textup{tet}}}{2}\sharp\{\mbox{non-simple homotopy classes of} \hspace{.2cm}  p^{-1}(\Gamma)\mbox{-arcs in} \hspace{.2cm} \Sigma_{1,1}\setminus\beta \} \\
     &\geq\frac{v_{\textup{tet}}}{2}(|\widetilde{A} |+|\widetilde{B}|).
 \end{align*}

Hence, we have
\[\mathrm{Vol}(\mathrm{T^1}\Sigma_{\textup{mod}} \setminus \widehat{\Gamma})\geq \frac{1}{6}\cdot\frac{v_{\textup{tet}}}{2}(|\widetilde{A} |+|\widetilde{B}|) = \frac{v_{\textup{tet}}}{12}(|\widetilde{A} |+|\widetilde{B}|).\]
\end{proof}

Recall that an \emph{unlabelled exponent} is an exponent viewed as a positive integer. An example illustrating the difference between unlabelled and labelled exponents can be found near the end of Section~\ref{Modulargeodesicscode}. 

\begin{corollary} 
\label{Cor:LowerBound}
Let $\widehat\Gamma$ be a modular link.  If $A$ is the set of all distinct (unlabelled) $L$-exponents and $B$ is the set of all distinct (unlabelled) $R$-exponents in the code words for $\widehat\Gamma$, then we have
 \[\mathrm{Vol}(\mathrm{T^1}\Sigma_{\textup{mod}} \setminus \widehat\Gamma)\geq \frac{v_{\textup{tet}}}{72} (|A|+|B|-10),\]
where $v_{\textup{tet}}\approx 1.01494$ is the volume of the regular ideal tetrahedron.  
\end{corollary} 

\begin{proof} 
We first construct the set $E_L$ algorithmically. Start with an empty set $E_L$ and list all labelled $L$-exponents in ascending order. Go through the list and pick the first labelled $L$-exponent that is greater than or equal to $6$ (if it exists), denote it by $\widetilde{a}_{1}$, and put it in $E_L$. For the $k$-th step, denote by $\widetilde{a}_{k}$ the labelled $L$-exponent that is at least $6$ greater than $\widetilde{a}_{k-1}$ and put it in $E_L$. Stop on the $N$-th step, where $N\leq |\widetilde{A}|$, when there does not exist a labelled $L$-exponent that is at least $6$ greater than $\widetilde{a}_{N-1}$. The algorithm will terminate since $|\widetilde{A}|$ is finite. 

There are at most $5$ distinct unlabelled $L$-exponents that were not chosen in the algorithm before the choice of $\widetilde{a}_{1}$. For $1 \leq k \leq |E_L|-1$, there are $6$ distinct unlabelled exponents that are greater than or equal to $\widetilde{a}_{k}$ and less than $\widetilde{a}_{k+1}$. For the last labelled $L$-exponent chosen in the algorithm, there are at most $6$ distinct unlabelled $L$-exponents greater than or equal to $\widetilde{a}_{|E_L|}$. Thus, we have $ (5+6|E_L|) \geq |A|$. Hence, we have 
\[|E_L|\geq \frac{|A|-5}{6}.\]
Note that the above inequality also holds when $E_L$ is empty because the non-existence of labelled $L$-exponent that is greater than or equal to $6$ implies $|A| \leq 5$. 

We define $E_R$ similarly and obtain the inequality \[|E_R|\geq \frac{|B|-5}{6}.\]

By \refdef{AtildeBtilde}, we have $|\widetilde{A}| + |\widetilde{B}|\geq |E_L|+|E_R|$. Therefore, we have
\[|\widetilde{A}| + |\widetilde{B}|\geq \frac{|A|+|B|-10}{6}.\] 

It follows from \refthm{LowerBound1} that 
\[\mathrm{Vol}(\mathrm{T^1}\Sigma_{\textup{mod}} \setminus \widehat{\Gamma})\geq \frac{v_{\textup{tet}}}{12}(|\widetilde{A} |+|\widetilde{B}|)\geq \frac{v_{\textup{tet}}}{72} (|A|+|B|-10).\]
\end{proof}
\vskip .2cm

The following corollary shows that an infinite family of modular knot complements has volume coarsely linear in terms of the word period. Note that the word period is the same as the trip number of a modular knot, and the trip number is the same as the braid index of the modular knot  \cite{Franks-Williams:BraidsAndJonesPolynomial, Birman-Williams:KnottedPeriodicOrbitsInDynamicalSystemsI, Birman-Kofman:NewTwistOnLorenzLinks}. 

\begin{corollary} \label{Cor:LinearBound} 
Let $n, i\in\ZZ_+$.  Let $\widehat\gamma$ be a modular knot with code word $L^{a_1} R^i  L^{a_2} R^i \ldots L^{a_n} R^i$, where $n$ is the word period of $\widehat\gamma$.  If \[a_1 > a_2 > \ldots > a_{n-1} > a_n \quad \text{and \quad for each } j\in\ZZ\cap[2,n], \ a_{j-1}-a_j \geq 6,  \] then 
 \[\frac{v_{\textup{tet} }}{12}n\leq \mathrm{Vol}(\mathrm{T^1}\Sigma_{\textup{mod}}\setminus\widehat\gamma) \leq 8 v_{\textup{tet}} (7n+2),\]
where $v_{\textup{tet}}\approx 1.01494$ is the volume of the regular ideal tetrahedron. 
\end{corollary} 
\vskip .2cm

\begin{proof}
    The statement follows from \cite[Theorem 6.2]{Hui-Rodriguez:BunchesUpperVolume} and \refthm{LowerBound1}.  
\end{proof}
\vskip .2cm

\section{Classification of modular link complements} \label{Sec:Classification} 

In \cite{Hui-Rodriguez:BunchesUpperVolume}, the notion of \emph{base orders} was defined to capture the notion of relative magnitudes of word exponents.  All modular link complements that share the same $x$-base order and same $y$-base order have the same parent link up to ambient isotopy. They thus share the same parent manifold and the same upper volume bound. The letter $L$ (or resp. $R$) in this paper corresponds to the letter $x$ (or resp. $y$) in \cite{Hui-Rodriguez:BunchesUpperVolume}. 
\vskip .2cm

The following result shows that the classification result in \cite[Theorem~5.3]{Hui-Rodriguez:BunchesUpperVolume} can be generalised to include lower volume bounds as well. 

\begin{theorem}[Classification] \label{Thm:Classification}
    All modular link complements can be partitioned into classes of members that have the same base orders of their corresponding modular links. All modular link complements in each class share the same upper volume bound and the same lower volume bound. 
\end{theorem} 

\begin{proof} 
    The result on upper volume bounds follows from \cite[Theorem~5.3]{Hui-Rodriguez:BunchesUpperVolume}.  We will show that all members in the same class share the same lower volume bound below. 

    Note that the $x$-base order of a modular link with $n$ components is a collection of $n$ ordered tuples, each of which describes the descending order of the labelled $x$-exponents in a word. The number of nonempty entries in each ordered tuple is the number of distinct $x$-exponents. A similar statement holds for $y$-base order. 

    Suppose two modular links share the same $x$-base order. As the set cardinalities are the same, the modular links have the same number of link components.  There is a one-to-one correspondence between the two sets of link components (or two sets of words) such that the number of nonempty entries in each ordered tuple for one word is the same as that for the corresponding word of the other link. Thus, if two modular links have the same $x$-base order, then they have the same number of distinct $x$-exponents. Similarly, if two modular links have the same $y$-base order, then they have the same number of distinct $y$-exponents.  

    By \refcor{LowerBound}, all members in the same class share the same lower volume bound. 
\end{proof}

%% %% To make the bibliography:  
\bibliographystyle{amsplain}  %% Uses AMS format for bibliography
%% Put all the bib entries in a file references.bib
\bibliography{biblio_MathSciNet.bib} 
\end{document}